\documentclass[a4paper, 11pt]{article}
\usepackage[T1]{fontenc}
\usepackage[utf8]{inputenc}
\usepackage{amsmath}
\usepackage{amsfonts}
\usepackage{amssymb}
\usepackage{amsthm}
\usepackage{graphicx}
\usepackage[english]{babel}
\usepackage{version} 
\usepackage{nicefrac}
\usepackage{tipa} 
\usepackage{xcolor}
\usepackage{mathtools}
\usepackage{mathrsfs}
\usepackage{hyperref}

\theoremstyle{definition}
\newtheorem{defin}{Definition}[section]

\theoremstyle{plain}
\newtheorem{theo}[defin]{Theorem}
\newtheorem{lemma}[defin]{Lemma}
\newtheorem{obs}[defin]{Remark}
\newtheorem{prop}[defin]{Proposition}

\newtheorem*{theorem-no-number}{Theorem}
\newtheorem{theorem}{Theorem}

\renewenvironment{abstract}
{\par\noindent\textbf{\abstractname.}\ \ignorespaces}
{\par\medskip}

\title{Bishop-Jones' Theorem and the ergodic limit set}
\author{Nicola Cavallucci}
\date{}

\begin{document}
\maketitle
\begin{abstract}
	\footnotesize
	For a proper, Gromov-hyperbolic metric space and a discrete, non-elementary, group of isometries, we define a natural subset of the limit set at infinity of the group called the ergodic limit set. The name is motivated by the fact that every ergodic measure which is invariant for the geodesic flow on the quotient metric space is concentrated on geodesics with endpoints belonging to the ergodic limit set.
	We refine the classical Bishop-Jones' Theorem proving that the packing dimension of the ergodic limit set coincides with the critical exponent of the group.
\end{abstract}
\tableofcontents

\section{Introduction}
The critical exponent of a discrete group of isometries of a proper metric space, defined as
\begin{equation}
	\label{eq:defin_critical_exponent}
	h_\Gamma := \limsup_{T \to +\infty} \frac{1}{T} \log \#(\Gamma x \cap B(x,T)),
\end{equation}
is a widely studied invariant, especially in case of negatively curved spaces. The classical and celebrated Bishop-Jones' Theorem relates $h_\Gamma$ to fine analytical properties of the boundary at infinity of $\Gamma$ if $X$ is Gromov-hyperbolic. It states what follows.

\begin{theo}[\cite{BJ97}, \cite{Pau97}, \cite{DSU17}]
	\label{theo:intro-Bishop-Jones}
	Let $X$ be a proper, $\delta$-hyperbolic metric space and let $\Gamma < \textup{Isom}(X)$ be non-elementary and discrete. Then
	$$h_\Gamma = \textup{HD}(\Lambda_{\textup{rad}}).$$
\end{theo}

We briefly explain the terms appearing in Theorem \ref{theo:intro-Bishop-Jones}, we refer to Sections \ref{sec:Gromov_hyperbolic}, \ref{sec-HP-dimensions}, \ref{sec-Bishop_Jones} for more details. Every $\Gamma$ as in the statement defines a limit set $\Lambda$, which is the set of accumulation points on the boundary at infinity $\partial X$ of $X$ of the set $\Gamma x$, with $x\in X$ fixed. This set does not depend on the choice of $x$ and it is the smallest closed $\Gamma$-invariant subset of $\partial X$. The boundary $\partial X$ of $X$ admits several visual metrics $D_{x,a}$ depending on the choice of a point $x\in X$ and a parameter $a > 0$. Given a subset $Y\subseteq \partial X$, one can computes the classical notions of fractal dimensions of $Y$ with respect to all these metrics. It turns out that, denoting for instance by $\text{HD}_{D_{x,a}}(Y)$ the Hausdorff dimension of $Y$ computed with respect to the metric $D_{x,a}$, then $a\cdot \text{HD}_{D_{x,a}}(Y) = b \cdot \text{HD}_{D_{x',a'}}(Y)$ for every admissible value of $a$ and $a'$ and every choice of $x$ and $x'$. This common value is simply denoted by $\text{HD}(Y)$ and it is called the generalized Hausdorff dimension of the set $Y$. In Section \ref{sec-HP-dimensions} we will see a natural construction of $\text{HD}(\cdot)$ via generalized visual balls. A similar construction, with similar properties as above, holds for other notions of dimensions, allowing us to define the generalized Minkowski dimension $\text{MD}(\cdot)$ and the generalized packing dimension $\text{PD}(\cdot)$. We refer to Section \ref{sec-HP-dimensions} for more details. \\
By definition, every point $z$ of the limit set $\Lambda$ of $\Gamma$ is the limit of a sequence of orbit points $\{g_i x\}_{i\in \mathbb{N}}$. However this sequence can converge to $z$ in different ways. A point $z \in \partial X$ is called radial if there exists a geodesic ray $\xi$ and a sequence $\{g_i x\}_{i\in \mathbb{N}}$ converging to $z$ such that $\sup_{i\in \mathbb{N}}d(\xi, g_i x) < \infty$. The set of all radial points, denoted by $\Lambda_{\text{rad}}$, appears in Theorem \ref{theo:intro-Bishop-Jones}. In particular the critical exponent of $\Gamma$, as defined in \ref{eq:defin_critical_exponent}, coincides with the generalized Hausdorff dimension of the radial limit set. In Theorem \ref{Bishop-Jones} we will recall the beautiful improvement of \cite{DSU17}, stating, among the other things, that one can find a smaller subset $\Lambda_{\text{u-rad}}$ of $\Lambda_{\text{rad}}$, called the set of uniformly radial limit points, for which the equality in Theorem \ref{theo:intro-Bishop-Jones} still holds. 

On the other hand one might wonder if the conclusion of Theorem \ref{theo:intro-Bishop-Jones} continue to hold if we replace the generalized Hausdorff dimension with other fractal dimensions. This is not possible for the generalized Minkowski dimension since $\text{MD}(\Lambda_{\text{u-rad}}) = \text{MD}(\Lambda_{\text{rad}}) = \text{MD}(\Lambda)$ because $\Lambda$ is the closure of the other two sets and it is known that generically $\text{MD}(\Lambda) > h_\Gamma$.
Indeed in \cite{DPPS09} there is an example of a pinched negatively curved Riemannian manifold $(M,g)$ admitting a non-uniform lattice $\Gamma$ (i.e. the volume of $\Gamma\backslash M$ is finite) such that $h_\Gamma < h_\text{vol}(M)$, where $h_\text{vol}(M)$ is the volume entropy of $M$. Since $\Gamma$ is a lattice we have $\Lambda = \partial M$, so $\text{MD}(\Lambda) = h_\text{vol}(M) > h_\Gamma$ by \cite[Theorem B]{Cav21}.

Concerning the packing dimension, our main finding is the following contribution to Theorem \ref{theo:intro-Bishop-Jones}.
\begin{theorem}
	\label{theo:intro-packing-ergodic-set}
	Let $X$ be a proper, $\delta$-hyperbolic metric space and let $\Gamma < \textup{Isom}(X)$ be non-elementary and discrete. Then
	$$h_\Gamma = \textup{PD}(\Lambda_{\textup{erg}}).$$
\end{theorem}

Here $\Lambda_{\text{erg}}$, called the ergodic limit set, is a subset satisfying $\Lambda_{\text{u-rad}} \subseteq \Lambda_{\text{erg}} \subseteq \Lambda_{\text{rad}}$. Its precise definition will be given in Section \ref{sec-Bishop_Jones}. The name will be explained in a moment. Before that we report that the same techniques used for the proof of Theorem \ref{theo:intro-packing-ergodic-set}, actually a simplified version of them, will be used to prove that the limit superior in \eqref{eq:defin_critical_exponent} is a true limit, generalizing Roblin's result (cp. \cite{Rob02}) holding for CAT$(-1)$ spaces.
\begin{theorem}
	\label{theo-intro-Roblin}
	Let $X$ be a proper, $\delta$-hyperbolic metric space and let $\Gamma < \textup{Isom}(X)$ be discrete and non-elementary. Then
	$$\limsup_{T \to +\infty} \frac{1}{T} \log \#(\Gamma x \cap B(x,T)) = \liminf_{T \to +\infty} \frac{1}{T} \log \#(\Gamma x \cap B(x,T)) = h_\Gamma.$$
\end{theorem}

Let us come back to the motivation behind the name of the ergodic limit set $\Lambda_{\text{erg}}$. It is related to the geodesic flow on the quotient metric space $\Gamma \backslash X$. To be precise we denote by $\text{Geod}(X)$ the space of geodesic lines of $X$. The group $\Gamma$ acts by homeomorphisms on $\text{Geod}(X)$ and the quotient is denoted by $\textup{Proj-Geod}(X)$. For instance, it coincides with the space of local geodesics of $\Gamma\backslash X$ when $X$ is CAT$(0)$ and $\Gamma$ is torsion-free, see Remark \ref{rem:local-geodesics}. The natural action of $\mathbb{R}$ by time reparametrizations $\Phi_t$ on $\text{Geod}(X)$ descends to a well-defined flow $\Phi_t$ on $\textup{Proj-Geod}(X)$, called the geodesic flow. In the study of the dynamical system $(\textup{Proj-Geod}(\Gamma\backslash X), \Phi_1)$ it is classically relevant to study  $\Phi_1$-invariant probability measures that are ergodic.
The next result motivates the name of the ergodic limit set. 

\begin{theorem}
	\label{theo-intro:ergodic_measures_concentrated_on_ergodic_limit_set}
	Let $X$ be a proper, geodesic, $\delta$-hyperbolic space. Let $\Gamma < \textup{Isom}(X)$ be discrete. 
	Let $\mu$ be an ergodic, $\Phi_1$-invariant, probability measure on $\textup{Proj-Geod}(\Gamma\backslash X)$. Then $\mu$ is concentrated on the set of equivalence classes of geodesics with endpoints belonging to $\Lambda_{\text{erg}}$.
\end{theorem}

The results of Theorem \ref{theo:intro-packing-ergodic-set} and Theorem \ref{theo-intro:ergodic_measures_concentrated_on_ergodic_limit_set} will be used in \cite{Cav24} to provide another proof of \cite[§6 \& Remark 6.1]{DilsavorThompson2023}. Indeed, in case $X$ is CAT$(-1)$, the packing dimension of $\Lambda_\text{erg}$ is naturally related to the entropy of a measure as in the statement of Theorem \ref{theo-intro:ergodic_measures_concentrated_on_ergodic_limit_set}.

\section{Gromov-hyperbolic spaces}
\label{sec:Gromov_hyperbolic}
Let $(X,d)$ be a metric space. The open (resp.closed) ball of radius $r$ and center $x \in X$ is denoted by $B(x,r)$ (resp. $\overline{B}(x,r)$). If we need to specify the metric we will write $B_d(x,r)$ (resp. $\overline{B}_d(x,r))$.
A geodesic segment is an isometric embedding $\gamma\colon I \to X$ where $I=[a,b] \subseteq \mathbb{R}$ is a bounded interval. The points $\gamma(a), \gamma(b)$ are called the endpoints of $\gamma$. A metric space $X$ is called geodesic if for every couple of points $x,y\in X$ there exists a geodesic segment whose endpoints are $x$ and $y$. Every such geodesic segment will be denoted, with an abuse of notation, by $[x,y]$. A geodesic ray is an isometric embedding $\xi\colon[0,+\infty)\to X$ while a geodesic line is an isometric embedding $\gamma\colon \mathbb{R}\to X$. 
\vspace{1mm}

\noindent Let $X$ be a geodesic metric space and let $x,y,z \in X$. The {\em Gromov product} of $y$ and $z$ with respect to $x$ is defined as
\vspace{-3mm}

$$(y,z)_x = \frac{1}{2}\big( d(x,y) + d(x,z) - d(y,z) \big).$$

\noindent The space $X$ is called {\em $\delta$-hyperbolic} if   for every $x,y,z,w \in X$   the following {\em 4-points condition} hold:
\begin{equation}\label{hyperbolicity}
	(x,z)_w \geq \min\lbrace (x,y)_w, (y,z)_w \rbrace -  \delta 
\end{equation}

\vspace{-2mm}
\noindent  or, equivalently,
\vspace{-5mm}

\begin{equation}
	\label{four-points-condition}
	d(x,y) + d(z,w) \leq \max \lbrace d(x,z) + d(y,w), d(x,w) + d(y,z) \rbrace + 2\delta. 
\end{equation}

\noindent The space $X$ is   {\em Gromov hyperbolic} if it is $\delta$-hyperbolic for some $\delta \geq 0$. \\
We recall that Gromov-hyperbolicity should be considered as a negative-curvature condition at large scale: for instance every CAT$(\kappa)$ metric space, with $\kappa <0$ is $\delta$-hyperbolic for a constant $\delta$ depending only on $\kappa$. The converse is false, essentially because the CAT$(\kappa)$ condition controls the local geometry much better than the Gromov-hyperbolicity due to the convexity of the distance functions in such spaces (see for instance \cite{LN19}, \cite{CavS20} and \cite{CavS20bis}).

\subsection{Gromov boundary}
\label{subsubsec-boundary}
Let $X$ be a proper, $\delta$-hyperbolic metric space and let $x$ be a point of $X$. \\
The {\em Gromov boundary} of $X$ is defined as the quotient 
$$\partial X = \lbrace (z_n)_{n \in \mathbb{N}} \subseteq X \hspace{1mm} | \hspace{1mm}   \lim_{n,m \to +\infty} (z_n,z_m)_{x} = + \infty \rbrace \hspace{1mm} /_\sim,$$
where $(z_n)_{n \in \mathbb{N}}$ is a sequence of points in $X$ and $\sim$ is the equivalence relation defined by $(z_n)_{n \in \mathbb{N}} \sim (z_n')_{n \in \mathbb{N}}$ if and only if $\lim_{n,m \to +\infty} (z_n,z_m')_{x} = + \infty$.  \linebreak
We will write $ z = [(z_n)] \in \partial X$ for short, and we say that $(z_n)$ {\em converges} to $z$. This definition  does not depend on the basepoint $x$.
There is a natural topology on $X\cup \partial X$ that extends the metric topology of $X$. 

\noindent Every geodesic ray $\xi$ defines a point  $\xi^+=[(\xi(n))_{n \in \mathbb{N}}]$  of the Gromov boundary $ \partial X$: we  say that $\xi$ {\em joins} $\xi(0) = y$ {\em to} $\xi^+ = z$. Moreover for every $z\in \partial X$ and every $x\in X$ it is possible to find a geodesic ray $\xi$ such that $\xi(0)=x$ and $\xi^+ = z$. Indeed if $(z_n)$ is a sequence of points converging to $z$ then, by properness of $X$, the sequence of geodesics $[x,z_n]$ subconverges to a geodesic ray $\xi$ which has the properties above (cp. \cite[Lemma III.3.13]{BH09}). A geodesic ray joining $x$ to $z\in \partial X$ will be denoted by $\xi_{x,z}$ or simply $[x,z]$.
The relation between Gromov product and geodesic ray is highlighted in the following lemma.
\begin{lemma}[\text{\cite[Lemma 4.2]{Cav21ter}}]
	\label{product-rays}
	Let $X$ be a proper, $\delta$-hyperbolic metric space, $z,z'\in \partial X$, $x\in X$, $b>0$. Then
	\begin{itemize}
		\item[(i)] if $(z,z')_{x} \geq T$ then $d(\xi_{x,z}(T - \delta),\xi_{x,z'}(T - \delta)) \leq 4\delta$;
		\item[(ii)] if $d(\xi_{x,z}(T),\xi_{x,z'}(T)) < 2b$ then $(z,z')_{x} > T - b$.
	\end{itemize}
\end{lemma}

\vspace{2mm}
\noindent The following is a standard computation, see \cite[Proposition 8.10]{BCGS} for instance.
\begin{lemma}
	\label{parallel-geodesics}
	Let $X$ be a proper, $\delta$-hyperbolic metric space. Then every two geodesic rays $\xi, \xi'$ with same endpoints at infinity are at distance at most $8\delta$, i.e. there exist $t_1,t_2\geq 0$ such that $t_1+t_2=d(\xi(0),\xi'(0))$ and  $d(\xi(t + t_1),\xi'(t+t_2)) \leq 8\delta$ for all $t\in \mathbb{R}$.
\end{lemma}

\subsection{Visual metrics}
\label{subsubsec-visual-metrics}
When $X$ is a proper, $\delta$-hyperbolic metric space it is known that the boundary $\partial X$ is metrizable. A metric $D_{x,a}$ on $\partial X$ is called a {\em visual metric} of center $x \in X$ and parameter $a\in\left(0,\frac{1}{2\delta\cdot\log_2e}\right)$ if there exists $V> 0$ such that for all $z,z' \in \partial X$ it holds
\begin{equation}
	\label{visual-metric}
	\frac{1}{V}e^{-a(z,z')_{x}}\leq D_{x,a}(z,z')\leq V e^{-a(z,z')_{x}}.
\end{equation}
For every $a$ as before and every $x\in X$ there exists a visual metric of parameter $a$ and center $x$, see \cite{Pau96}.
As in \cite{Pau96} and \cite{Cav21ter} we define the {\em generalized visual ball} of center $z \in \partial X$ and radius $\rho \geq 0$ as
$$B(z,\rho) = \bigg\lbrace z' \in \partial X \text{ s.t. } (z,z')_{x} > \log \frac{1}{\rho} \bigg\rbrace.$$
It is comparable to the metric balls of the visual metrics on $\partial X$.
\begin{lemma}
	\label{comparison-balls}
	Let $D_{x,a}$ be a visual metric of center $x$ and parameter $a$ on $\partial X$. Then for every $z\in \partial X$ and for every $\rho>0$ it holds
	$$B_{D_{x,a}}\left(z, \frac{1}{V}\rho^a\right)  \subseteq B(z,\rho)\subseteq B_{D_{x,a}}(z, V\rho^a ).$$
\end{lemma}
\noindent It is classical that generalized visual balls are related to shadows, whose definition is the following. Let $x\in X$ be a basepoint. The shadow of radius $r>0$ casted by a point $y\in X$ with center $x$ is the set:
$$\text{Shad}_x(y,r) = \lbrace z\in \partial X \text{ s.t. } [x,z]\cap B(y,r) \neq \emptyset \text{ for every ray } [x,z]\rbrace.$$
For our purposes we just need:
\begin{lemma}[\text{\cite[Lemma 4.8]{Cav21ter}}]
	\label{shadow-ball}
	Let $X$ be a proper, $\delta$-hyperbolic metric space. Let $z\in \partial X$, $x\in X$ and $T\geq 0$. Then for every $r>0$ it holds
	$$\textup{Shad}_{x}\left(\xi_{x,z}\left(T\right), r\right) \subseteq B(z, e^{-T + r}).$$
\end{lemma}


\section{Hausdorff and Packing dimensions}
\label{sec-HP-dimensions}
In this section we recall briefly the definitions of Hausdorff and packing dimensions of a subset of a metric space. Then we will adapt these constructions and results to the case of the boundary at infinity of a $\delta$-hyperbolic metric space. The facts presented here are classical and can be found easily in literature.

\subsection{Definitions of Hausdorff and Packing dimensions}
Let $(X,d)$ be a metric space and $\alpha \geq 0$. The $\alpha$-Hausdorff measure of a Borel subset $B\subset X$ is defined as
$$\mathcal{H}^\alpha_d(B) = \lim_{\eta \to 0}\inf \left\lbrace \sum_{i\in \mathbb{N}} r_i^\alpha \text{ s.t. } B\subseteq \bigcup_{i\in \mathbb{N}}B(x_i,r_i) \text{ and } r_i\leq \eta\right\rbrace.$$
The argument of the limit is increasing when $\eta$ tends to $0$, so the limit exists. This formula actually defines a Borel measure on $X$.
To be precise what we introduced is the definition of the spherical Hausdorff measure. It is comparable to the classical Hausdorff measure. The Hausdorff dimension of a Borel subset $B$ of $X$, denoted HD$_d(B)$, is the unique real number $\alpha \geq 0$ such that $\mathcal{H}^{\alpha'}_d(B) = 0$ for every $\alpha' > \alpha$ and $\mathcal{H}^{\alpha'}_d(B) = +\infty$ for every $\alpha' < \alpha$. 

\vspace{1mm}
\noindent The packing dimension is defined in a similar way, but using disjoint balls inside $B$ instead of coverings. For every $\alpha \geq 0$ and for every Borel subset $B$ of $X$ we define
$$\mathcal{P}^\alpha_d(B) = \lim_{\eta \to 0} \sup\left\lbrace \sum_{i\in \mathbb{N}} r_i^\alpha \text{ s.t. } B(x_i,r_i) \text{ are disjoint, }x_i\in B \text{ and } r_i\leq \eta\right\rbrace.$$
This is not a measure on $X$ but only a pre-measure. By a standard procedure one can define the $\alpha$-Packing measure as
$$\hat{\mathcal{P}}_d^\alpha(B) = \inf\left\lbrace \sum_{k=1}^\infty \mathcal{P}_d^\alpha(B_k) \text{ s.t. } B\subseteq \bigcup_{k=1}^\infty B_k, \, B_k \text{ Borel}\right\rbrace.$$
The packing dimension of a Borel subset $B\subseteq X$, denoted PD$_d(B)$, is the unique real number $\alpha \geq 0$ such that $\hat{\mathcal{P}}^{\alpha'}_d(B) = 0$ for every $\alpha' > \alpha$ and $\hat{\mathcal{P}}^{\alpha'}_d(B) = +\infty$ for every $\alpha' < \alpha$. \\
The packing dimension has another useful interpretation (cp \cite[Proposition 3.8]{Fal04}): for every Borel subset $B\subseteq X$ we have
\begin{equation}
	\label{eq-pack-MD}
	\text{PD}_d(B) = \inf\left\lbrace \sup_k \overline{\text{MD}}_d(B_k) \text{ s.t. } B\subseteq \bigcup_{k=1}^\infty B_k,\, B_k \text{ Borel}\right\rbrace.
\end{equation}
The quantity $\overline{\text{MD}}_d$ denotes the upper Minkowski dimension, namely:
\begin{equation}
	\label{MD-definition}
	\overline{\text{MD}}_d(B) = \limsup_{r \to 0}\frac{\log\text{Cov}_d(B,r)}{\log\frac{1}{r}},
\end{equation}
where $B$ is any subset of $X$ and $\text{Cov}_d(B,r)$ denotes the minimal number of $d$-balls of radius $r$ needed to cover $B$. Taking the limit inferior in place of the limit superior in \eqref{MD-definition} one defines the lower Minkowski dimension of $B$, denoted $\underline{\text{MD}}_d(B)$.

\subsection{Visual dimensions}
Let $X$ be a proper, $\delta$-hyperbolic metric space and let $x\in X$. The boundary at infinity $\partial X$ supports several visual metrics $D_{x,a}$, so the Hausdorff dimension, the packing dimension and the Minkowski dimension of subsets of $\partial X$ are well defined with respect to $D_{x,a}$. There is a way to define universal versions of these quantities that do not depend neither on $x$ nor on $a$.
Fix $\alpha \geq 0$. For a Borel subset $B$ of $\partial X$ we set, following \cite{Pau96},
$$\mathcal{H}^\alpha(B) = \lim_{\eta \to 0}\inf \left\lbrace \sum_{i\in \mathbb{N}} \rho_i^\alpha \text{ s.t. } B\subseteq \bigcup_{i\in \mathbb{N}}B(z_i,\rho_i) \text{ and } \rho_i\leq \eta\right\rbrace,$$
where $B(z_i,\rho_i)$ are generalized visual balls.
As in the classical case the {\em{ visual Hausdorff dimension}} of $B$ is defined as the unique $\alpha \geq 0$ such that $\mathcal{H}^{\alpha'}(B) = 0$ for every $\alpha' > \alpha$ and $\mathcal{H}^{\alpha'}(B) = +\infty$ for every $\alpha'<\alpha$. The visual Hausdorff dimension of the Borel subset $B$ is denoted by HD$(B)$. By Lemma \ref{comparison-balls}, see also \cite{Pau96}, we have HD$(B) = a\cdot\text{HD}_{D_{x,a}}(B)$ for every visual metric $D_{x,a}$ of center $x$ and parameter $a$.
\vspace{2mm}

\noindent In the same way we can define the visual $\alpha$-packing pre-measure of a Borel subset $B$ of $\partial X$ by
$$\mathcal{P}^\alpha(B) = \lim_{\eta \to 0} \sup\left\lbrace \sum_{i\in \mathbb{N}} \rho_i^\alpha \text{ s.t. } B(z_i,\rho_i) \text{ are disjoint, }x_i\in B \text{ and } \rho_i\leq \eta\right\rbrace,$$
where $B(z_i,\rho_i)$ are again generalized visual balls. As usual we can define the visual $\alpha$-packing measure by
$$\hat{\mathcal{P}}^\alpha(B) = \inf\left\lbrace \sum_{k=1}^\infty \mathcal{P}^\alpha(B_k) \text{ s.t. } B\subseteq \bigcup_{k=1}^\infty B_k,\, B_k \text{ Borel}\right\rbrace.$$
Consequently it is defined the visual packing dimension of a Borel set $B$, denoted by PD$(B)$. Using Lemma \ref{comparison-balls} as in the case of the Hausdorff measure (see \cite{Pau96}) one can check that for every visual metric $D_{x,a}$ of center $x$ and parameter $a$ it holds:
$$\frac{1}{V^a} \hat{\mathcal{P}}_{D_{x,a}}^{\frac{\alpha}{a}}(B) \leq \hat{\mathcal{P}}^\alpha(B) \leq V^a \hat{\mathcal{P}}_{D_{x,a}}^{\frac{\alpha}{a}}(B)$$
for every $\alpha \geq 0$ and every Borel subset $B\subseteq \partial X$. Therefore for every Borel set $B$ it holds PD$(B)=a\cdot \text{PD}_{D_{x,a}}(B)$.
\vspace{2mm}

\noindent Using generalized visual balls, instead of metric balls with respect to a visual metric, one can define the visual upper and lower Minkowski dimension of a subset $B\subseteq \partial X$:
$$\overline{\text{MD}}(B) = \limsup_{\rho \to 0}\frac{\log \text{Cov}(B,\rho)}{\log \rho}, \qquad \underline{\text{MD}}(B) = \liminf_{\rho \to 0}\frac{\log \text{Cov}(B,\rho)}{\log \rho},$$
where $\text{Cov}(B,\rho)$ denotes the minimal number of generalized visual balls of radius $\rho$ needed to cover $B$. 
Using again Lemma \ref{comparison-balls} one has $\overline{\text{MD}}(B) = a\cdot  \overline{\text{MD}}_{D_{x,a}}(B)$ for every Borel set $B$ and every visual metric of center $x$ and parameter $a$. The same holds for the lower Minkowski dimension.
\vspace{2mm}

\noindent It is easy to check that for every Borel set $B$ of $\partial X$ the numbers HD$(B)$, PD$(B)$, $\underline{\text{MD}}(B)$, $\overline{\text{MD}}(B)$ do not depend on $x$, see \cite[Proposition 6.4]{Pau96}, and their definition is independent also on $a$.
Using the classical facts holding for metric spaces we get
\begin{equation}
	\text{HD}(B) \leq \text{PD}(B) \leq \underline{\text{MD}}(B) \leq \overline{\text{MD}}(B)
\end{equation}
and 
\begin{equation}
	\label{PD-supMD}
	\text{PD}(B) = \inf\left\lbrace \sup_k \overline{\text{MD}}(B_k) \text{ s.t. } B\subseteq \bigcup_{k=1}^\infty B_k,\,B_k \text{ Borel}\right\rbrace.
\end{equation}
for every Borel subset $B$ of $\partial X$.
\vspace{2mm}

\section{Limit sets of discrete groups of isometries}
\label{sec-Bishop_Jones}
If $X$ is a proper metric space we denote its group of isometries by Isom$(X)$ and we endow it with the uniform convergence on compact subsets of $X$. A subgroup $\Gamma$ of Isom$(X)$ is called {\em discrete} if the following equivalent conditions hold:
\begin{itemize}
	\item[(a)] $\Gamma$ is discrete as a subspace of Isom$(X)$; \vspace{-2mm}
	\item[(b)] $\forall x\in X$ and $R\geq 0$ the set $\Sigma_R(x) = \lbrace g \in \Gamma  \text{ s.t. }  g  x\in \overline{B}(x,R)\rbrace$ is finite.
\end{itemize}   

\noindent The critical exponent of a discrete group of isometries $\Gamma$ acting on a proper metric space $X$ can be defined using the Poincaré series, or alternatively (\cite{Cav21ter}, \cite{Coo93}), as
$$\overline{h_\Gamma}(X) = \limsup_{T \to +\infty}\frac{1}{T}\log \# (\Gamma x \cap B(x,T)),$$
where $x$ is a fixed point of $X$. This quantity does not depend on the choice of $x$. In the following we will often write $\overline{h_\Gamma}(X)=:h_\Gamma$.
Taking the limit inferior instead of the limit superior we define the lower critical exponent, denoted by $\underline{h_\Gamma}(X)$. In \cite{Rob02} it is proved that if $\Gamma$ is a discrete, non-elementary group of isometries of a CAT$(-1)$ space then $\overline{h_\Gamma}(X) = \underline{h_\Gamma}(X)$. Theorem \ref{theo-intro-Roblin} generalizes this result to proper, $\delta$-hyperbolic spaces.
\vspace{1mm}

\noindent We specialize the situation to the case of a proper, $\delta$-hyperbolic metric space $X$. Every isometry of $X$ acts naturally on $\partial X$ and the resulting map on $X\cup \partial X$ is a homeomorphism.
The {\em limit set} $\Lambda(\Gamma)$ of a discrete group of isometries $\Gamma$ is the set of accumulation points of the orbit $\Gamma x$ on $\partial X$, where $x$ is any point of $X$; it is the smallest $\Gamma$-invariant closed set of the Gromov boundary (cp. \cite[Theorem 5.1]{Coo93}) and it does not depend on $x$.\\
There are several interesting subsets of the limit set: the radial limit set, the uniformly radial limit set, etc. They are related to important sets of the geodesic flow on the quotient space $\Gamma \backslash X$. We will see an instance in the second part of the paper. In order to recall their definiton we need to introduce a more general class of subsets of $\partial X$.\\
We fix a basepoint $x\in X$. Let $\tau$ and $\Theta = \lbrace \vartheta_i \rbrace_{i\in \mathbb{N}}$ be, respectively, a positive real number and an increasing sequence of real numbers with $\lim_{i \to +\infty}\vartheta_i = +\infty$. We define
$\Lambda_{\tau, \Theta}(\Gamma)$ as the set of points $z\in \partial X$ such that there exists a geodesic ray $[x,z]$ satisfying the following: for every $i\in \mathbb{N}$ there exists a point $y_i \in [x,z]$ with $d(x,y_i) \in [\vartheta_i, \vartheta_{i+1}]$ such that $d(y_i,\Gamma x) \leq \tau$. We observe that up to change $\tau$ with $\tau + 8\delta$ the definition above does not depend on the choice of the geodesic ray $[x,z]$, by Lemma \ref{parallel-geodesics}.
\begin{lemma}
	In the situation above it holds:
	\begin{itemize}
		\item[(i)] $\Lambda_{\tau, \Theta}(\Gamma) \subseteq \Lambda(\Gamma)$;
		\item[(ii)] the set $\Lambda_{\tau, \Theta}(\Gamma)$ is closed. 
	\end{itemize}
\end{lemma}
\begin{proof}
	The first statement is obvious, so we focus on (ii). Let $z^k \in \Lambda_{\tau, \Theta}(\Gamma)$ be a sequence converging to $z^\infty$. Let $\xi^k = [x,z^k]$ be a geodesic ray as in the definition of $\Lambda_{\tau, \Theta}(\Gamma)$. We know that, up to a subsequence, the sequence $\xi^k$ converges uniformly on compact sets of $[0,+\infty)$ to a geodesic ray $\xi^\infty = [x,z^\infty]$. We fix $i\in \mathbb{N}$ and we take points $y_i^k$ with $d(x,y_i^k)\in [\vartheta_{i}, \vartheta_{i+1}]$ and $d(y_i^k,\Gamma x)\leq \tau$. The sequence $y_i^k$ converges to a point $y_i^\infty \in [x,z^\infty]$ with $d(x,y_i^\infty) \in [\vartheta_{i}, \vartheta_{i+1}]$. Moreover clearly $d(y_i^\infty, \Gamma x) \leq \tau$. Since this is true for every $i\in \mathbb{N}$ we conclude that $z^\infty \in \Lambda_{\tau, \Theta}(\Gamma)$.
\end{proof}
\noindent We can now introduce some interesting subsets of the limit set of $\Gamma$. Let $\Theta_\text{rad}$ be the set of increasing, unbounded sequences of real numbers.
The \emph{radial limit set} is classically defined as 
$$\Lambda_\text{rad}(\Gamma) = \bigcup_{\tau \geq 0}\bigcup_{\Theta \in \Theta_\text{rad}}\Lambda_{\tau, \Theta}(\Gamma).$$
The \emph{uniform radial limit set} is defined (see \cite{DSU17}) as
$$\Lambda_\text{u-rad}(\Gamma) = \bigcup_{\tau \geq 0}\Lambda_{\tau}(\Gamma),$$
where $\Lambda_\tau(\Gamma)=\Lambda_{\tau, \lbrace i\tau\rbrace}(\Gamma)$.\\
Another interesting set that is what we call the \emph{ergodic limit set}, defined as:
$$\Lambda_\text{erg}(\Gamma) = \bigcup_{\tau \geq 0}\bigcup_{\Theta \in \Theta_\text{erg}}\Lambda_{\tau, \Theta}(\Gamma),$$
where a sequence $\Theta = \lbrace\vartheta_i\rbrace$ belongs to $\Theta_\text{erg}$ if $\exists\lim_{i\to +\infty} \frac{\vartheta_{i}}{i} \in (0,+\infty)$. 
The name is justified by Theorem \ref{theo-intro:ergodic_measures_concentrated_on_ergodic_limit_set} stating that every ergodic measure which is invariant by the geodesic flow on $\Gamma\backslash X$ is concentrated on geodesics whose endpoints belong to $\Lambda_{\text{erg}}$.\\
When $\Gamma$ is clear in the context, we will simply write $\Lambda_{\tau, \Theta}, \Lambda_{\text{rad}}, \Lambda_{\text{u-rad}}, \Lambda_{\textup{erg}}, \Lambda$, omitting $\Gamma$.
\begin{lemma}
	\label{lemma:gamma-invariance-of-limit-sets}
	In the situation above the sets $\Lambda_{\textup{rad}}, \Lambda_{\textup{u-rad}}$ and $\Lambda_{\textup{erg}}$ are $\Gamma$-invariant and do not depend on $x$.
\end{lemma}
\begin{proof}
	Let $y$ be another point of $X$ and let $z\in \partial X$. By Lemma \ref{parallel-geodesics} for every couple of geodesic rays $\xi = [y,z]$, $\xi' = [x,z]$ there are $t_1,t_2\geq 0$ such that $t_1+t_2\leq d(x,y)$ and $d(\xi(t+t_1), \xi'(t+t_2))\leq 8\delta$. This means that $d(\xi(t), \xi'(t)) \leq d(x,y) + 8\delta$ for every $t\geq 0$. It is then straightforward to see that if $z\in \Lambda_{\tau, \Theta}$ (as defined with respect to $x$) then it belongs to $\Lambda_{\tau + d(x,y) + 8\delta, \Theta}$ as defined with respect to $y$. This shows the thesis.
\end{proof}

\section{Bishop-Jones' Theorem}
\label{subsec:BJ_proof}
The celebrated Bishop-Jones' Theorem, in the general version of \cite{DSU17}, states the following:
\begin{theo}[\cite{BJ97}, \cite{Pau97}, \cite{DSU17}]
	\label{Bishop-Jones}
	Let $X$ be a proper, $\delta$-hyperbolic metric space and let $\Gamma < \textup{Isom}(X)$ be discrete and non-elementary. Then
	$$h_\Gamma = \textup{HD}(\Lambda_{\textup{rad}}) = \textup{HD}(\Lambda_{\textup{u-rad}}) = \sup_{\tau \geq 0} \textup{HD}(\Lambda_{\tau}).$$
\end{theo}
%
%
%
In order to introduce the techniques we will use in the proof of Theorem \ref{theo:intro-packing-ergodic-set} we start with the
\begin{proof}[Proof of Theorem \ref{theo-intro-Roblin}]
	By Theorem \ref{Bishop-Jones} we have 
	$$\overline{h_{\Gamma}}(X) = h_\Gamma = \sup_{\tau \geq 0} \text{HD}(\Lambda_\tau) \leq \sup_{\tau \geq 0} \underline{\text{MD}}(\Lambda_\tau).$$
	So it would be enough to show that 
	$$\sup_{\tau \geq 0} \underline{\text{MD}}(\Lambda_\tau) \leq \underline{h_\Gamma}(X).$$
    We fix $\tau \geq 0$. For every $\varepsilon > 0$ we take a subsequence $T_j \to +\infty$ such that 
	$$\frac{1}{T_j}\log \#(\Gamma x \cap \overline{B}(x,T_j)) \leq \underline{h_\Gamma}(X) + \varepsilon$$
	for every $j$.
	We define $\rho_j = e^{-T_j}$: notice that $\rho_j \to 0$. 
	Let $k_j\in \mathbb{N}$ be such that $(k_j-1)\tau \le T_j < k_j\tau$. If
	$z\in \Lambda_\tau$ then there exists a geodesic ray $[x,z]$ and a point $y_j \in [x,z]$ with $d(x,y_j) \in [(k_j-3)\tau, (k_j-2)\tau]$ and $d(y_j, gx) \le \tau$ for some $g\in \Gamma$. This $g$ satisfies $d(x,gx)\le (k_j-1)\tau \le T_j$. 
	Moreover $z\in \text{Shad}_x(gx, \tau + 8\delta)$, since $d(gx,[x,z]) \le \tau$ and since every two parallel geodesic rays are $8\delta$ apart by Lemma \ref{parallel-geodesics}. We showed that the set of shadows $\{\text{Shad}_x(gx, \tau + 8\delta)\}$ with $g\in \Gamma$ such that $(k_j-4)\tau \le d(x,gx)\le (k_j-1)\tau \le T_j$ cover $\Lambda_\tau$. The cardinality of this set of shadows is at most $e^{(\underline{h_\Gamma}(X) + \varepsilon)T_j} \le e^{(\underline{h_\Gamma}(X) + \varepsilon)k_j\tau}$.
	Among these shadows indexed by these elements $g\in \Gamma$ we select the ones that intersect $\Lambda_\tau$. For these ones, the construction above gives a point $z_g \in \Lambda_\tau$, a point $y_g$ along $[x,z_g]$ such that $(k_j-3)\tau \le d(x,y_g) \le (k_j -2)\tau$ and $d(y_g,gx)\le \tau$. Therefore
	\begin{equation}
		\begin{aligned}
			\text{Shad}_{x}(gx, \tau + 8\delta)\subseteq \text{Shad}_{x}(y_g, 2\tau + 8\delta) &\subseteq B(z_g, e^{2\tau+8\delta}e^{-d(x,y_g)})\\ &\subseteq B(z_g, e^{5\tau+8\delta}\rho_j),
		\end{aligned}
	\end{equation}
	by Lemma \ref{shadow-ball}. This shows that $\Lambda_\tau$ is covered by at most $e^{(\underline{h_\Gamma}(X) + \varepsilon)k_j\tau}$ generalized visual balls of radius $e^{5\tau+8\delta}\rho_j$.
	Therefore
	\begin{equation*}
		\begin{aligned}
		\underline{\text{MD}}(\Lambda_{\tau})&\leq \liminf_{j \to +\infty}\frac{\log\text{Cov}(\Lambda_{\tau}, e^{5\tau + 8\delta}\rho_j)}{\log \frac{1}{e^{5\tau + 8\delta}\rho_j}} \\
		&\leq \liminf_{j \to +\infty}\frac{(\underline{h_\Gamma}(X) + \varepsilon)k_j\tau}{-5\tau - 8\delta + (k_j-1)\tau} = \underline{h_\Gamma}(X)+\varepsilon.
		\end{aligned}
	\end{equation*}
	By the arbitrariness of $\varepsilon$ we conclude the proof.
\end{proof}
	\noindent There are several remarks we can do about this proof:
	\begin{itemize}
		\item[(a)] The proof is still valid for every sequence $T_j \to + \infty$, so it implies also that $\sup_{\tau \geq 0} \overline{\textup{MD}}(\Lambda_\tau) \leq h_\Gamma$. Therefore we have another improvement of Bishop-Jones Theorem, namely:
		\begin{equation}
			\label{eq-MD-critical}
			\sup_{\tau \geq 0} {\textup{HD}}(\Lambda_\tau) = \sup_{\tau \geq 0} \underline{\textup{MD}}(\Lambda_\tau) = \sup_{\tau \geq 0} \overline{\textup{MD}}(\Lambda_\tau)=h_\Gamma.
		\end{equation}
		\item[(b)] $\Lambda_{\textup{u-rad}} = \bigcup_{\tau \in \mathbb{N}}\Lambda_\tau$, so by (a) and \eqref{PD-supMD} we deduce that $\textup{PD}(\Lambda_{\textup{u-rad}})=h_\Gamma$.
		\item[(c)] We can get the same estimate of the Minkowski dimensions from above weakening the assumptions on the sets $\Lambda_\tau$. Indeed take a set $\Lambda_{\tau, \Theta}$ such that $\limsup_{i\to +\infty} \frac{\vartheta_{i+1}}{\vartheta_i} = 1$. Then we can cover this set by shadows casted by points of the orbit $\Gamma x$ whose distance from $x$ is between $\vartheta_{i_j}$ and $\vartheta_{i_j + 1}$, with $i_j \to +\infty$ when $j \to + \infty$.
		Therefore arguing as before we obtain
		$$\underline{\textup{MD}}(\Lambda_{\tau, \Theta}) \leq \liminf_{j \to +\infty}\frac{(\underline{h_\Gamma}(X) + \varepsilon)\vartheta_{i_j + 1}}{\vartheta_{i_j - 1}} \leq \underline{h_\Gamma}(X)+\varepsilon,$$
		where the last step follows by the asymptotic behaviour of the sequence $\Theta$. A similar estimate holds for the upper Minkowski dimension.
		\item[(d)] One could be tempted to conclude that the packing dimension of the set $\bigcup_{\tau \geq 0}\bigcup_{\Theta} \Lambda_{\tau, \Theta}$, where $\Theta$ is a sequence such that $\limsup_{i\to +\infty} \frac{\vartheta_{i+1}}{\vartheta_i} = 1$, is $\leq h_\Gamma$. But this is not necessarily true since in \eqref{PD-supMD} it is required a countable covering and not an arbitrary covering. That is why the estimate of the packing dimension of the ergodic limit set $\Lambda_\textup{erg}$ in Theorem \ref{theo:intro-packing-ergodic-set} is not so easy. However as we will see in a moment the ideas behind the proof are similar to the ones used in the proof of  Theorem \ref{theo-intro-Roblin}.
	\end{itemize}

\begin{proof}[Proof of Theorem \ref{theo:intro-packing-ergodic-set}]
We notice it is enough to prove that PD$(\Lambda_\text{erg}) \leq h_\Gamma$. The strategy is the following: for every $\varepsilon > 0$ we want to find a countable family of sets $\lbrace B_k\rbrace_{k\in\mathbb{N}}$ of $\partial X$ such that $\Lambda_{\text{erg}} \subseteq \bigcup_{k=1}^\infty B_k$ and $\sup_{k\in \mathbb{N}}\overline{\text{MD}}(B_k)\leq (h_\Gamma + \varepsilon)(1+\varepsilon)$. Indeed if this is true then by \eqref{PD-supMD}:
$$\text{PD}(\Lambda_\text{erg}) \leq \sup_{k\in \mathbb{N}}\overline{\text{MD}}(B_k)\leq (h_\Gamma + \varepsilon)(1+\varepsilon),$$
and by the arbitrariness of $\varepsilon$ the thesis is true.\\
So we fix $\varepsilon > 0$ and we proceed to define the countable family. For $m,n\in \mathbb{N}$ and $l\in \mathbb{Q}_{> 0}$ we define
$$B_{m,l,n} = \bigcup_{\Theta} \Lambda_{m,\Theta},$$
where $\Theta$ is taken among all sequences such that for every $i\geq n$ it holds 
$$l-\eta_l \leq \frac{\vartheta_{i}}{i} \leq l + \eta_l,$$
where $\eta_l = \frac{\varepsilon}{2+\varepsilon}\cdot l$.\\
First of all if $z\in \Lambda_\text{erg}$ we know that $z\in \Lambda_{m,\Theta}$ for some $m \in \mathbb{N}$ and $\Theta$ satisfying $\lim_{i\to +\infty}\frac{\vartheta_{i}}{i} = L \in (0,\infty)$, in particular there exists $n\in \mathbb{N}$ such that $L - \beta \leq \frac{\vartheta_{i}}{i} \leq L + \beta$ for every $i\geq n$, where $\beta = \frac{2+\varepsilon}{4+3\varepsilon}\cdot\eta_L$. Now we take $l \in \mathbb{Q}_{>0}$ such that $\vert L - l \vert < \beta$. Then it is easy to see that $[L-\beta,L + \beta]\subseteq [l-2\beta, l+ 2\beta]$ and $\eta_l \geq \eta_L - \frac{\varepsilon}{2+\varepsilon}\beta \geq 2\beta$. So by definition $z\in B_{m,l,n}$, therefore $\Lambda_\text{erg}\subseteq \bigcup_{m,l,n} B_{m,l,n}.$ \\
Now we need to estimate the upper Minkowski dimension of each set $B_{m,l,n}$. We take $T_0$ big enough such that 
$$\frac{1}{T}\log \#(\Gamma x \cap \overline{B}(x,T)) \leq h_\Gamma + \varepsilon$$
for every $T\geq T_0$.
Let us fix $\rho \leq e^{-\max\lbrace T_0, n(l-\eta_l) \rbrace}$. We consider $j\in \mathbb{N}$ with the following property: $(j-1)(l-\eta_l) < \log \frac{1}{\rho}\leq j(l-\eta_l)$. We observe that the condition on $\rho$ gives $\log \frac{1}{\rho} \geq n(l-\eta_l)$, implying $j\geq n$.\\
We consider the set of elements $g\in \Gamma$ such that 
\begin{equation}
	\label{shadow}
	j(l-\eta_l) - m \leq d(x,gx) \leq (j+1)(l+\eta_l) + m.
\end{equation}
For any such $g$ we consider the shadow $\text{Shad}_{x}(gx,2m + 8\delta)$. We claim that this set of shadows covers $B_{m,l,n}$. 
Indeed every point $z$ of $B_{m,l,n}$ belongs to some $\Lambda_{m,\Theta}$ with $l-\eta_l \leq \frac{\vartheta_{i}}{i} \leq l + \eta_l$ for every $i\geq n$. In particular this holds for $i=j$, and so $j(l-\eta_l) \leq \vartheta_j \leq j(l+\eta_l)$. Hence there exists a point $y$ along a geodesic ray $[x,z]$ satisfying:
$$j(l-\eta_l)\leq \vartheta_j \leq d(x,y)\leq \vartheta_{j+1}\leq (j+1)(l+\eta_l), \qquad d(y,\Gamma x) \leq m.$$
So there is $g\in \Gamma$ satisfying \eqref{shadow} such that $z\in \text{Shad}_{x}(gx,2m + 8\delta)$, by Lemma \ref{parallel-geodesics}.
Moreover these shadows are casted by points at distance at least $j(l-\eta_l) - m$ from $x$, so at distance at least $\log\frac{1}{e^m\rho}$ from $x$. We need to estimate the number of such $g$'s.
By the assumption on $\rho$ we get that this number is less than or equal to $e^{(h_\Gamma+\varepsilon)[(j+1)(l+\eta_l) + m]}.$
Hence, using again Lemma \ref{shadow-ball}, we conclude that $B_{m,l,n}$ is covered by at most $e^{(h_\Gamma+\varepsilon)[(j+1)(l+\eta_l) + m]}$ generalized visual balls of radius $e^{5m+8\delta}\rho$. Thus
%
\begin{equation*}
	\begin{aligned}
		\overline{\text{MD}}(B_{m,l,n})&=\limsup_{\rho \to 0}\frac{\log\text{Cov}(B_{m,l,n}, e^{5m+8\delta}\rho)}{\log \frac{1}{e^{5m+8\delta}\rho}}\\
		&\leq \limsup_{j\to +\infty}\frac{(h_\Gamma+\varepsilon)[(j+1)(l+\eta_l) + m]}{-5m -8\delta + (j-1)(l-\eta_l)} \\
		&\leq (h_\Gamma+\varepsilon)(1+\varepsilon),
	\end{aligned}
\end{equation*}
where the last inequality follows from the choice of $\eta_l$.
\end{proof}

\section{An interpretation of the ergodic limit set}
\label{sec-dynamics}

Let $X$ be a proper metric space. The {\em space of parametrized geodesic lines} of $X$ is
$$\text{Geod}(X) = \lbrace \gamma\colon \mathbb{R} \to X \text{ isometric embedding}\rbrace,$$
considered as a subset of $C^0(\mathbb{R},X)$, the space of continuous maps from $\mathbb{R}$ to $X$ endowed with the uniform convergence on compact subsets of $\mathbb{R}$. By lower semicontinuity of the length under uniform convergence (cp. \cite[Proposition I.1.20]{BH09}), we have that $\text{Geod}(X)$ is closed in $C^0(\mathbb{R},X)$. 
There is a natural action of $\mathbb{R}$ on $\text{Geod}(X)$ defined by reparametrization:
$$\Phi_t\gamma (\cdot) = \gamma(\cdot + t)$$
for every $t\in \mathbb{R}$.
It is a continuous action, i.e. the map $\Phi_t$ is a homeomorphism of $\text{Geod}(X)$ for every $t\in\mathbb{R}$ and $\Phi_t \circ \Phi_s = \Phi_{t+s}$ for every $t,s\in \mathbb{R}$. This action is called the \emph{geodesic flow} on $X$.

Let $\Gamma$ be a discrete group of isometries of $X$. We consider the quotient space $\Gamma \backslash X$ and the standard projection $\pi\colon X \to \Gamma \backslash X$. On the quotient it is defined a standard pseudometric by $d(\pi x, \pi y) = \inf_{g\in \Gamma}d(x, gy)$. Since the action is discrete then this pseudometric is actually a metric. Indeed if $d(\pi x, \pi y)= 0$ then for every $n > 0$ there exists $g_n\in \Gamma$ such that $d(x,g_ny)\leq \frac{1}{n}$. In particular $d(x,g_nx)\leq d(x,g_ny) + d(g_ny,g_nx) \leq d(x,y) + 1$ for every $n$. The cardinality of these $g_n$'s is finite, thus there must be one of these $g_n$'s such that $d(x, g_ny) = 0$, i.e. $x= g_ny$, and so $\pi x = \pi y$. 

The group $\Gamma$ acts on $\textup{Geod}(X)$ by $(g\gamma)(\cdot) = g(\gamma(\cdot))$. This action is by homeomorphisms and we define the space
$$\textup{Proj-Geod}(\Gamma\backslash X) := \Gamma \backslash \textup{Geod}(X),$$
endowed with the quotient topology. The elements of $\text{Proj-Geod}(\Gamma\backslash X)$ will be denoted by $[\gamma]$, where $\gamma \in \text{Geod}(X)$ is a representative. The action of $\Gamma$ commutes with the flow $\Phi_t$ in the sense that $g\circ \Phi_t = \Phi_t \circ g$ for every $g\in \Gamma$ and $t\in \mathbb{R}$. Therefore the flow $\Phi_t$ defines a flow on $\textup{Proj-Geod}(\Gamma\backslash X)$, i.e. an action of $\mathbb{R}$ by homeomorphisms. This flow, still denoted $\Phi_t$, is called the \emph{geodesic flow} on $\Gamma \backslash X$. 

\begin{obs}
	\label{rem:local-geodesics}
	The name is a bit improper in this generality. Indeed \linebreak $\textup{Proj-Geod}(\Gamma\backslash X)$ does not coincide with the space of local geodesics of $\Gamma \backslash X$. However, when $\Gamma$ acts freely then every element of $\textup{Proj-Geod}(\Gamma\backslash X)$ is a local geodesic of $\Gamma \backslash X$. If, additionally, every local geodesic of $X$ is a geodesic, then $\textup{Proj-Geod}(\Gamma\backslash X)$ is naturally homeomorphic to the space of local geodesics of $\Gamma \backslash X$. In this case the flow on $\textup{Proj-Geod}(\Gamma\backslash X)$ coincides with the geodesic flow on the space of all local geodesics of $\Gamma \backslash X$. The assumptions above are satisfied for instance when $X$ is Busemann convex (e.g. \textup{CAT}$(0)$) and $\Gamma$ is torsion-free. Observe that the space $\textup{Proj-Geod}(\Gamma\backslash X)$ is the one studied also in \cite{DilsavorThompson2023} in the \textup{CAT}$(-1)$ setting.
\end{obs}

	The couple $(\textup{Proj-Geod}(\Gamma\backslash X), \Phi_1)$, where $\Phi_1$ is the geodesic flow of $\Gamma \backslash X$ at time $1$, is a dynamical system. An important role in its study is played by $\Phi_1$-invariant probability measures, i.e. Borel measures $\mu$ on $\textup{Proj-Geod}(\Gamma\backslash X)$ with total mass $1$ and such that $(\Phi_1)_\#\mu = \mu$, where $(\Phi_1)_\#$ denotes the pushforward. The set of $\Phi_1$-invariant probability measures is a closed, convex subset of all Borel measures on $\textup{Proj-Geod}(\Gamma\backslash X)$, whose extremal points are ergodic. We recall that a $\Phi_1$-invariant probability measure is ergodic if for every $\Phi_1$-invariant subset $A\subseteq \textup{Proj-Geod}(\Gamma\backslash X)$, i.e. such that $\Phi_1^{-1}(A) = \Phi_{-1}(A) \subseteq A$, we have $\mu(A) \in \{0,1\}$.
	Ergodic measures satisfy the famous Birkhoff's Ergodic Theorem that we now state in our specific situation.
	\begin{prop}
		\label{prop-Birkhoff}
		Let $X$ be a proper metric space, let $\Gamma < \textup{Isom}(X)$ be discrete. Let $(\textup{Proj-Geod}(\Gamma\backslash X), \Phi_1)$ be the geodesic flow on $\Gamma \backslash X$ as defined above. Let $\mu$ be an ergodic, $\Phi_1$-invariant probability measure. For every $f \in L^1(\mu)$ it holds
		\begin{equation}
			\label{Birkhoff}
			\lim_{N\to +\infty}\frac{1}{N}\sum_{j=0}^{N-1} (f \circ \Phi_j)([\gamma]) = \int f\,d\mu
		\end{equation}
		for $\mu$-a.e. $[\gamma] \in \textup{Proj-Geod}(\Gamma\backslash X)$. In other words, the limit in \eqref{Birkhoff} exists for $\mu$-a.e. $[\gamma] \in \textup{Proj-Geod}(\Gamma\backslash X)$ and equals the right hand side.
	\end{prop}

	The next result, which is a reformulation of Theorem \ref{theo-intro:ergodic_measures_concentrated_on_ergodic_limit_set}, motivates the name of the ergodic limit set. 

\begin{theo}
	\label{theo:ergodic_measures_concentrated_on_ergodic_limit_set}
	Let $X$ be a proper, $\delta$-hyperbolic space. Let $\Gamma < \textup{Isom}(X)$ be discrete and non-elementary. 
	Let $\mu$ be an ergodic, $\Phi_1$-invariant, probability measure on $\textup{Proj-Geod}(\Gamma\backslash X)$. Then $\mu$ is concentrated on the set
	$$\{[\gamma] \in \textup{Proj-Geod}(\Gamma\backslash X) \, : \, \gamma^\pm \in \Lambda_\textup{erg}\}.$$
\end{theo}
Notice that the property $\gamma^\pm \in \Lambda_\textup{erg}$ is well defined, i.e. it does not depend on the representative of the class $[\gamma]$. This follows by the $\Gamma$-invariance of $\Lambda_\textup{erg}$, see Lemma \ref{lemma:gamma-invariance-of-limit-sets}.
\begin{proof}
	Since $X$ is proper we can find a countable set $\{x_i\}_{i\in\mathbb{N}}\subseteq X$ such that $X = \bigcup_{i\in\mathbb{N}} B(x_i,1)$. For every $i$ we define the sets
	$$V_i := \{ \gamma \in \text{Geod}(X)\,:\, \gamma(0) \in B(x_i,1) \}$$
	and
	$$U_i := \Gamma \backslash V_i \subseteq \text{Proj-Geod}(\Gamma\backslash X).$$
	Since $\{V_i\}_{i\in \mathbb{N}}$ is a covering of $\text{Geod}(X)$ then also $\{U_i\}_{i\in\mathbb{N}}$ is a covering of $\text{Proj-Geod}(\Gamma\backslash X)$. In particular there must be some $i_0\in \mathbb{N}$ such that $\mu(U_{i_0}) = c > 0$.
	To every $[\gamma] \in U_{i_0}$ we associate the set of integers $\Theta([\gamma]) = \lbrace \vartheta_i([\gamma])\rbrace$ defined recursively by
	$$\vartheta_0([\gamma])=0, \qquad \vartheta_{i+1}([\gamma]) = \min \lbrace n\in\mathbb{N}, n > \vartheta_i([\gamma]) \text{ s.t. } \Phi_n([\gamma]) \in U_{i_0}\rbrace.$$
	We apply Proposition \ref{prop-Birkhoff} to the indicator function of the set $U_{i_0}$, namely $\chi_{U_{i_0}}$, obtaining that for $\mu$-a.e.$[\gamma] \in \text{Proj-Geod}(\Gamma\backslash X)$ it holds
	$$\exists \lim_{N\to +\infty}\frac{1}{N}\sum_{j=0}^{N-1} (\chi_{U_{i_0}} \circ \Phi_j)([\gamma]) = \mu(U_{i_0}) = c \in (0,1].$$
	We remark that $(\chi_{U_{i_0}} \circ \Phi_j)([\gamma]) = 1$ if and only if $j\in \Theta([\gamma])$ and it is $0$ otherwise. So 
	$$\lim_{N\to +\infty}\frac{1}{N}\sum_{j=0}^{N-1} (\chi_{U_{i_0}} \circ \Phi_j)([\gamma]) = \lim_{N\to +\infty}\frac{\#\Theta([\gamma]) \cap [0,N-1]}{N},$$
	and the right hand side is by definition the density of the set $\Theta([\gamma])$. It is classical that, given the standard increasing enumeration $\lbrace \vartheta_0([\gamma]), \vartheta_1([\gamma]),\ldots \rbrace$ of $\Theta([\gamma])$, it holds 
	$$\lim_{N\to +\infty}\frac{\#\Theta([\gamma]) \cap [0,N-1]}{N} = \lim_{N\to +\infty}\frac{N}{\vartheta_N([\gamma])}.$$
	Putting all together we conclude that for $\mu$-a.e.$[\gamma] \in \text{Proj-Geod}(\Gamma \backslash X)$ the following is true
	\begin{equation}
		\label{eq-rec-times}
		\exists \lim_{N\to +\infty}\frac{\vartheta_N([\gamma])}{N} = \frac{1}{c} \in [1,+\infty).
	\end{equation}
	In the same way, applying the same argument to the flow at time $-1$ we get that for $\mu$-a.e.$[\gamma] \in \text{Proj-Geod}(\Gamma \backslash X)$ we have
	\begin{equation}
		\label{eq-rec-times-negative}
		\exists \lim_{N\to +\infty}\frac{\vartheta_N([-\gamma])}{N} = \frac{1}{c} \in [1,+\infty).
	\end{equation}
	Here $-\gamma$ denotes the curve $-\gamma(t) = \gamma(-t)$. We deduce that \eqref{eq-rec-times} and \eqref{eq-rec-times-negative} hold together for $\mu$-a.e.$[\gamma] \in \text{Proj-Geod}(\Gamma \backslash X)$. Finally we need to prove that for every $[\gamma] \in \text{Proj-Geod}(\Gamma \backslash X)$ satisfying \eqref{eq-rec-times} and \eqref{eq-rec-times-negative} we have $\gamma^{\pm} \in \Lambda_\text{erg}$. We show that $\gamma^+ \in \Lambda_\text{erg}$, the other being similar.
	We notice that an integer $n$ satisfies $n \in \Theta([\gamma])$ if and only if there exists a representative $g\gamma$ of $[\gamma]$, with $g\in \Gamma$, such that $\Phi_n(g\gamma) \in V_{i_0}$, i.e. $g\gamma(n) \in B(x_{i_0},1)$. In other words $n \in \Theta([\gamma])$ if and only if
	\begin{equation}
		\label{eq:returning_ergodic}
		d(\gamma(n), \Gamma x_{i_0}) < 1.
	\end{equation}
	We choose $x_{i_0}$ as basepoint of $X$. We fix a geodesic ray $\xi = [x_{i_0}, \gamma^+]$. By Lemma \ref{parallel-geodesics} we have that $d(\xi(t), \gamma(t)) \leq 8\delta + 1$ for every $t\geq 0$. This, together with \eqref{eq:returning_ergodic} says that $d(\xi(\vartheta_N([\gamma])), \Gamma x_{i_0}) < 8\delta + 2$. By definition this means that $\gamma^+ \in \Lambda_{\tau, \Theta([\gamma])}$, where $\tau = 8\delta + 2$. Finally we observe that the sequence $\Theta([\gamma])=\lbrace\vartheta_N([\gamma])\rbrace$ satisfies \eqref{eq-rec-times}, that is exactly the condition that defines a sequence involved in the definition of $\Lambda_\text{erg}$. Repeating the argument for $\gamma^-$, we get the thesis.
\end{proof}

\bibliographystyle{alpha}
\bibliography{An_extension_of_Bishop_Jones_Theorem}

\end{document}